\documentclass[12pt]{article}

\pdfoutput=1
\usepackage[margin=26mm]{geometry}
\usepackage{amsthm,amsmath,mathbbol,mathtools,amsfonts,amssymb,graphicx}
\usepackage{hyperref}
\usepackage{algorithm}
\usepackage{url}
\usepackage{epsfig}
\usepackage{wrapfig}
\usepackage{color}
\usepackage{blkarray}
 \usepackage[textsize=scriptsize]{todonotes}
\usepackage{lmodern}
\usepackage{bigstrut}
\usepackage{booktabs}
\usepackage{arydshln}
\usepackage{subfig}
\usepackage{algorithm}
\usepackage{algpseudocode}
\usepackage{comment}
\usepackage{pgf,tikz}
\usepackage[all]{xy}
\usetikzlibrary{matrix, positioning}
\usetikzlibrary{patterns}
\usepackage{tikz-cd}
\tikzset{
  big hook/.style = {
    hook',
    shorten <=2pt,
    line width=0.9pt
  }
}
\usepackage{enumitem}

\usetikzlibrary{arrows}
\usepackage{pbox}
\usepackage[capitalize]{cleveref}
\usepackage{xspace}
\usepackage{xstring}
\usepackage{multicol}

\usepackage{csquotes}

\newtheorem{theorem}{Theorem}
\newtheorem{lemma}[theorem]{Lemma}
\newtheorem{example}[theorem]{Example}
\newtheorem{proposition}[theorem]{Proposition}

\newtheorem{conjecture}[theorem]{Conjecture}
\newtheorem{definition}[theorem]{Definition}

\newtheorem*{corollary*}{Corollary}

\newcommand{\R}{\mathbb{R}}

\newcommand{\Pb}{\mathbb{P}}

\DeclareMathOperator{\Ker}{Ker}

\DeclareMathOperator{\Pru}{Pru}

\tikzset{
  big hook/.style = {hook, line width=1pt}
}

\numberwithin{theorem}{section}

\makeatletter
\renewcommand*\env@matrix[1][*\c@MaxMatrixCols c]{%
	\hskip -\arraycolsep
	\let\@ifnextchar\new@ifnextchar
	\array{#1}}
\makeatother

\title{Pruning distance of upset-decomposable persistence modules}
\date{}
\author{Roy Nicolas Nehme\thanks{LAGA, Université Sorbonne Paris Nord, CNRS, Université Paris 8\newline
Email: nehme@math.univ-paris13.fr}}

\begin{document}

\maketitle

 \begin{abstract}
 The pruning distance recently introduced by Bjerkevik compares persistence modules using approximate decompositions called prunings. Bjerkevik conjectures that this distance is Lipschitz equivalent to the classical interleaving distance on modules of a bounded pointwise dimension. In this article, we establish a Lipschitz equivalence with respect to the bottleneck distance for upset-decomposable persistence modules. In particular, this proves half of Bjerkevik's conjecture for these modules. More precisely, we bound the bottleneck distance by a multiple of the pruning distance, improving the conjectured bound from~$2r$ to~$(2r-1)$ where~$r$ is the maximal pointwise dimension, and show that this improved bound is sharp. We also prove the converse inequality, bounding the pruning distance by the bottleneck distance. Our approach relies on explicitly computing the pruning of upset-decomposable modules, which we carry out using a directed graph formalism.

\end{abstract}

\tableofcontents
\markboth{Table of Contents}{} 
\clearpage

\section{Introduction} 
\markboth{1. Introduction}{} 

Developed initially for data analysis, topological persistence provides invariants of filtered topological spaces through computation of their homology. This method produces algebraic objects called \emph{persistence modules}, which are formally functors from the ordered set~$(\R,\leq)$ to the category of vector spaces. A central result of the theory ensures that these modules decompose into simple elements associated with intervals when their vector spaces are finite-dimensional. The collection of intervals appearing in such a decomposition, the \emph{persistence barcode}, completely describes the original persistence module up to isomorphism and, therefore, contains fine information about the initial topological filtration.

The remarkable isometry theorem \cite{chazal_isometry} guarantees that the algebraically defined \emph{interleaving distance}~$d_I$ between modules coincides with the combinatorial \emph{bottleneck distance}~$d_B$ between their barcodes. The latter is defined through persistence diagrams and can be formulated as a minimum-cost bottleneck matching problem on a complete bipartite graph \cite{Edelsbrunner_Harer}, which can be solved in polynomial time \cite{Kerber_Morozov}. In particular, this theorem shows that the interleaving distance, while stable with respect to distance on data by construction, is also computable in polynomial time.

When dealing with multi-parameter data, it is natural to consider filtrations depending on several parameters instead of a single one, thus creating a \emph{multi-parameter persistence module}, that is, persistence modules indexed by~$\R^n$ endowed with the product order. However, although the interleaving distance and the bottleneck distance are both generalized to the multiparameter setting, the isometry theorem no longer holds: multiparameter modules are not described in a stable way by their indecomposables. So, the bottleneck distance fails to satisfy the stability property in the multi-parameter case, making its usage useless. Henceforth, the computational advantage of the interleaving distance disappears in the multi-parameter setting, in particular, computing the interleaving distance for multi-parameter persistence modules has been proven to be NP-hard \cite{bjerkevik_Botnan_kerber}. 

Bjerkevik introduced the \emph{$\varepsilon$-pruning} (Definition~\ref{def_pruning}) of a module in \cite{bjerkevik_stabilizing}. The~$\varepsilon$-pruning of~$M$ is an approximation of~$M$ which is particularly decomposable while discarding features below scale~$\varepsilon$. This notion of pruning allowed Bjerkevik to introduce the pruning distance~$d_P$ (Definition~\ref{def_pruning_distance}) which is related to both the interleaving distance and the decompositions. Therefore, it is unclear whether its computation is NP-hard or polynomial time computable. Bjerkevik conjectured that computing prunings can be done in polynomial time, and observes that, in the examples used to establish NP-hardness of computing the interleaving distance in \cite{bjerkevik_botnan}, the pruning distance appears to be computable in polynomial time, suggesting the possibility of a general polynomial time algorithm.

Importantly, Bjerkevik conjectured that the pruning distance is stable, and even Lipschitz equivalent to the interleaving distance, on persistence modules of a bounded pointwise dimension. More precisely, if~$M$ and~$N$ are persistence modules on an arbitrary poset with pointwise dimension smaller than a positive integer~$r$,
\[
d_P(M,N) \leq d_I(M,N) \leq 2r\,d_P(M,N).
\]

In this paper, we prove that the pruning distance is Lipschitz equivalent to the bottleneck distance for the class of upset-decomposable persistence modules, that is, modules that decompose as direct sums of interval persistence modules supported on subsets~$U \subset \R^n$ which are upward closed under the product order. Upset-decomposable modules are an important family of multi-parameter persistence modules that we study in detail in Section~\ref{sec:pruning-upset}. Our main results, Theorem~\ref{theorem_main} and Theorem~\ref{theorem_3}, show the following
\begin{theorem}
If~$M$ and~$N$ are two upset-decomposable modules of pointwise dimension at most~$r$, then
\[
d_P(M,N) \leq d_B(M,N) \leq (2r-1)d_P(M,N).
\]
\end{theorem}
This result improves the right-hand inequality in Bjerkevik’s conjecture by reducing the constant from~$2r$ to~$2r-1$, and strengthens it by replacing the interleaving distance with the bottleneck distance, which dominates the interleaving distance. Moreover, in Example~\ref{ex_2} we show that the constant~$2r-1$ is optimal on the class of upset-decomposable persistence modules, for  both the interleaving distance and the bottleneck distance. The left-hand inequality follows from our analysis and is weaker than the conjectured inequality involving the interleaving distance.

Our approach is based on an explicit computation of the pruning of upset-decomposable persistence modules, carried out in Theorem~\ref{theorem_1}. To compute the pruning, we associate a directed graph to an upset-decomposable module, thereby translate the problem into computing the reachable indices via the directed paths in the graph, which yields a purely combinatorial procedure.

\section{Preliminaries }
\markboth{2. Preliminaries}{}

\textbf{Remark.} 
All definitions and results presented in this chapter are taken from Bjerkevik's~\cite{bjerkevik_stabilizing}, and no proofs will be provided here as they can be found in that reference. We include them as prerequisites since the purpose of this paper is to prove an inequality of the Conjecture~\ref{conjecture} introduced in the same paper, in the case of upset-decomposable persistence modules. No part of this chapter constitutes original work, and readers familiar with~\cite{bjerkevik_stabilizing} may safely skip this chapter until Conjecture~\ref{conjecture}.

\setcounter{theorem}{0} 
\renewcommand{\thetheorem}{2.\arabic{theorem}}

\subsection{Persistence modules}

Let~$\Pb$ be a poset, which we identify with its poset category.
Our result will be in the poset~$\R^d,$ but since the pruning and the pruning distance were introduced by Bjerkevik in \cite{bjerkevik_stabilizing} in the general poset, we will give in this section the main definitions to get the general definition of the conjecture.

In the rest of this work we will assume that~$\mathbb{P}$ is equipped with a \emph{shift function}:
\[
\mathbb{P}_\varepsilon : \mathbb{P} \to \mathbb{P} \quad \text{for each } \varepsilon \in \mathbb{R},
\]
that is:

\begin{itemize}
    \item~$\mathbb{P}_0$ is the identity;
    \item For~$p \leq q \in \mathbb{P}$, we have~$\mathbb{P}_\varepsilon(p) \leq \mathbb{P}_\varepsilon(q)$;
    \item For all~$p \in \mathbb{P}$ and~$\varepsilon \geq 0$,~$\mathbb{P}_\varepsilon(p) \geq p$;
    \item For all~$\varepsilon, \delta \in \mathbb{R}$,~$\mathbb{P}_{\varepsilon + \delta} = \mathbb{P}_\varepsilon \circ \mathbb{P}_\delta$.
\end{itemize}

It follows from the first and fourth properties that all the functions~$\mathbb{P}_\varepsilon$ are bijections.  
In the rest of this paper, we will assume that~$\mathbb{P}$ satisfies the properties above.

Let~$\mathbf{Vec}$ denote the category of vector spaces over some fixed field~$k$.\\

\begin{definition}[]\label{def_persistence_module}
A \emph{persistence module} is a functor: \[M : \mathbb{P} \longrightarrow \mathbf{Vec},
\]

A \emph{morphism of persistence modules}~$f : M \to N$ is a natural transformation between the corresponding functors.
\end{definition}

Throughout this paper, the term module refers to a persistence module.

When the poset~$\mathbb{P}$ is taken to be~$\mathbb{R}^d$, with the shift function~$\mathbb{P}_\varepsilon(p) = p + (\varepsilon, \ldots, \varepsilon)$, we refer to these as \textit{$d$-parameter modules}. In particular, we distinguish between the case~$d = 1$, referred to as \textit{one-parameter modules}, and the case~$d > 1$, referred to as \textit{multi-parameter modules}.

Given a module~$M$, we write~$M_p$ the vector space associated to on an object~$p \in \mathbb{P}$, and~$M_{p \to q}$ for the morphism induced by~$p \leq q$ in the poset.

We say that a module~$M$ is \emph{pointwise finite-dimensional} (abbreviated as \textit{pfd}) if~$M_p$ is finite-dimensional for all~$p \in \mathbb{P}$.

Given two modules~$M$ and~$N$, the \emph{direct sum}~$M \oplus N$ is the module defined by
\[
(M \oplus N)_p := M_p \oplus N_p,
\]
with structure maps given component wise in the obvious way. A module~$M$ is called \emph{indecomposable} if, for every decomposition
\[
M = M_1 \oplus M_2,
\]
one of the summands~$M_1$ or~$M_2$ is the zero module.

The following theorem shows us that pointwise finite-dimensional modules have an essentially unique way of decomposing into, and allows us to introduce the subsequent definition.

\begin{theorem}[{\cite[Theorem 1.1]{botnan_crawley-boevey}, \cite[Theorem 1(ii)]{azumaya1950corrections}}]\label{th_barcode}
Let~$M$ be a pointwise finite-dimensional (pfd) module. Then there exists a set~$\{ M_i \}_{i \in I}$ of nonzero indecomposable modules with local endomorphism rings such that
\[
M \cong \bigoplus_{i \in I} M_i.
\]
Moreover, if~$M \cong \bigoplus_{j \in J} M'_j$, where each~$M'_j$ is nonzero and indecomposable, then there exists a bijection~$\sigma : I \to J$ such that~$M_i \cong M'_{\sigma(i)}$ for all~$i \in I$.
\end{theorem}

\begin{definition}[]\label{def_barcode}
Let~$M$ be a pointwise finite-dimensional (pfd) module isomorphic to a direct sum~$M \cong \bigoplus_{i \in I} M_i$, where each~$M_i$ is nonzero and indecomposable. \\
We define the \emph{barcode} of~$M$ to be the multiset
\[
B(M) := \{ M_i \}_{i \in I},
\]
where the~$M_i$ are considered isomorphism classes of modules.\\
By Theorem~\ref{th_barcode} the barcode is well defined, and is an isomorphism invariant.
\end{definition}

\begin{definition}[]\label{def_shift}
Let~$M$ be a module and~$\varepsilon \in \mathbb{R}$. The \emph{$\varepsilon$-shift} of~$M$, denoted~$M(\varepsilon)$, is the module defined by
\[
M(\varepsilon)_p := M_{\mathbb{P}_\varepsilon(p)} \quad \text{and} \quad M(\varepsilon)_{p \to q} := M_{\mathbb{P}_\varepsilon(p) \to \mathbb{P}_\varepsilon(q)}.
\]
For~$\varepsilon \leq \delta$, define the morphism
\[
M_{\varepsilon \to \delta} : M(\varepsilon) \to M(\delta)
\]
by~$(M_{\varepsilon \to \delta})_p := M_{\mathbb{P}_\varepsilon(p) \to \mathbb{P}_\delta(p)}$.
\end{definition}
\begin{definition}
Let~$M$ be a module.  
A \emph{submodule}~$N$ of~$M$, denoted as N~$\subseteq$ M, is module such that~$N_p \subseteq M_p$ for all~$p$, and for every~$p \le q$, the following diagram commutes:
\[
\begin{tikzcd}
N_p \arrow[r,"N_{p\to q}"] \arrow[d,hook] & N_q \arrow[d,hook] \\
M_p \arrow[r,"M_{p\to q}"] & M_q
\end{tikzcd}
\]
\end{definition}

\begin{definition}[]\label{def_subquotient}
Let~$M$ be a module. A module~$N$ is called a \emph{subquotient} of~$M$ if there exists submodules~$K \subseteq I \subseteq M$ such that~$N \cong I/K$.

Equivalently, A module~$N$ is a \emph{subquotient} of~$M$ if there exists a module~$M'$, along with a monomorphism~$f: M' \hookrightarrow M$ and an epimorphism~$g: M' \twoheadrightarrow N$. 

\end{definition}

\begin{lemma}\label{lem_subquotient}
The subquotient relation is a partial order on the set of isomorphism classes of pfd modules.
\end{lemma}

A distance on modules is always understood to be an extended pseudo-metric. Extended means the distance might be infinite and pseudo means we allow distance zero between different modules.

\begin{definition}\label{def_interleaving }
An~$\varepsilon$\emph{-interleaving} between two modules~$M$ and~$N$ is a pair of morphisms
\[
\varphi : M \to N(\varepsilon) \quad \text{and} \quad \psi : N \to M(\varepsilon)
\]
such that
\[
\psi(\varepsilon) \circ \varphi = M_{0 \to 2\varepsilon} \quad \text{and} \quad \varphi(\varepsilon) \circ \psi = N_{0 \to 2\varepsilon}.
\]
If such an interleaving exists, we say that~$M$ and~$N$ are~$\varepsilon$-interleaved and write~$M \sim_\varepsilon N$.
The \emph{interleaving distance}~$d_I$ is defined by
\[
d_I(M, N) = \inf \{ \varepsilon \geq 0 \mid M \sim_\varepsilon N \}.
\]
\end{definition}

\begin{definition}\label{def_bottleneck}
Let~$M$ and~$N$ be pfd modules. A \emph{matching} from~$M$ to~$N$ is a bijection
\[
\sigma : \bar{B}(M) \to \bar{B}(N)
\]
where~$\bar{B}(M) \subseteq B(M)$ and~$\bar{B}(N) \subseteq B(N)$.\\
If in addition:
\begin{itemize}
    \item for each~$Q \in \bar{B}(M)$,~$Q \sim_\varepsilon \sigma(Q)$.
    \item for each~$Q \in (B(M) \setminus \bar{B}(M)) \cup (B(N) \setminus \bar{B}(N))$,~$Q \sim_\varepsilon 0$ (the zero module).
\end{itemize}
then~$\sigma$ is called an \emph{$ \varepsilon$-matching} from~$M$ to~$N$.

The \emph{bottleneck distance}~$d_B$ is defined by:
\[
d_B(M, N) = \inf \{ \varepsilon \geq 0 \mid \text{there exists an } \varepsilon\text{-matching between } M \text{ and } N \}.
\]
\end{definition}

We will use the notation~$x \in M$ for a module~$M$ to mean that~$x \in M_p$ for some point~$p \in \mathbb{P}$.

\subsection{Prunings}

In order to address some difficulties arising in the proof of his main result in~\cite{bjerkevik_stabilizing}, Bjerkevik introduced the notion of \emph{$\varepsilon$-pruning}. Although it was originally motivated by these specific difficulties, pruning turned out to possess robust theoretical properties. In this section, we will recall its definition, its construction and an example of computing a pruning.

\begin{definition}\label{def_pruning}
\emph{(Pruning)} Let~$M$ be a module. Let~$I$ be the largest submodule of~$M$ such that for any morphism~$f: M \to M(2\varepsilon)$, we have~$f(I) \subseteq M_{0 \to 2\varepsilon}(I)$. \\
Let~$K$ be the smallest submodule of~$I$ such that for any morphism~$f: M \to M(2\varepsilon)$, it holds that~$I^{-1}_{0 \to 2\varepsilon}(f(K)) \subseteq K$.\\
We define the \emph{$\varepsilon$-pruning pair} of~$M$ to be the pair~$(I, K)$, and the \emph{$\varepsilon$-pruning} of~$M$ to be 
\[
\Pru_\varepsilon(M) := (I/K)(-\varepsilon).
\]
\textit{For example} we have that the~$0$-pruning pair of a module~$M$ is~$(M, 0)$, and the~$0$-pruning of~$M$ is 
\[
\Pru_0(M) = (M/0)(0) \cong M.
\]
\end{definition}

\begin{lemma}\label{lem_pruning_construction}
Let~$M$ be a module. Then the~$\delta$-pruning pair~$(I, K)$ of~$M$ is well-defined.

Let~$\{ f_j \}_{j \in \Gamma}$ be the set of all morphisms from~$M$ to~$M(2\delta)$. Explicitly:
\begin{itemize}
  \item~$I = \bigcap_{i=0}^\infty I_i$, where we define
  \[
  I_0 = M, \quad I_i = \bigcap_{j \in \Gamma} f_j^{-1}(M_{0 \to 2\delta}(I_{i-1})) \subseteq M \quad \text{for all } i \geq 1.
  \]
  
  \item~$K = \bigcup_{i=0}^\infty K_i$, where we define
  \[
  K_0 = 0, \quad K_i = \sum_{j \in \Gamma} I^{-1}_{0 \to 2\delta}(f_j(K_{i-1})) \subseteq I \quad \text{for all } i \geq 1.
  \]
\end{itemize}
Moreover, for all~$i \geq 0$, we have~$I_i \supseteq I_{i+1}$ and~$K_i \subseteq K_{i+1}$.
\end{lemma}

\textbf{Example of computing the pruning.}

Let~$M$ be a 1-parameter module defined as~$M = M_1 \oplus M_2 \oplus M_3$,  
with~$M_1 = [0,8]$,~$M_2 = [2,8]$, and~$M_3 = [4,8]$, where each~$M_i$ is an interval module in the standard 1-parameter setting, which we do not define explicitly in this paper.\\
Let us calculate~$\Pru_1(M)$:\\
Using {\cite[Proposition 5.3]{bauer_lesnick}}, we have that if~$f: M \to N$ is a morphism of 1-parameter p.f.d. persistence modules and~$f$ maps~$[b,d] \text{ into } [b',d'].$
Then 
$
b' \le b < d' \le d.
$\\
So a morphism~$f$ from~$M$ to~$M(2)$ can map~$M_1$ into~$M_1(2)$ or~$M_2(2)$, and can map~$M_2$ or~$M_3$ into~$M_1(2)$ or~$M_2(2)$ or~$M_3(2)$. We choose~$f_1$ be a morphism from~$M$ to~$M(2)$ such that it maps:~$M_1 \mapsto M_2(2)$,~$M_2 \mapsto M_3(2)$,~$M_3 \mapsto M_3(2)$. Since~$\operatorname{supp}(M_3) \subseteq \operatorname{supp}(M_2) \subseteq \operatorname{supp}(M_1)$, the construction of~$f_1$ was made that its preimage to be included in the preimage of any other morphism~$f$ from~$M$ to~$M(2)$.
We set~$I_0 = M$.
\[
M_{0 \to 2}(I_0) \cong [0,6] \oplus [2,6] \oplus [4,6],
\]
Then:
\[ I_1 \cong \bigcap_{j \in \Gamma} f_j^{-1}(M_{0 \to 2\delta}(I_0)) \cong f_1^{-1}(M_{0 \to 2}(I_0))\cong [2,8] \oplus [4,8] \oplus [4,8].
\]
To compute~$I_2$, we apply the same reasoning:
\[
I_2 \cong f_1^{-1}(M_{0 \to 2}(I_1)) \cong [4,8] \oplus [4,8] \oplus [4,8].
\]
We observe that~$I_3 \cong I_2$, so the iteration has stabilized.\\
\[
K \cong K_1 \cong \Ker(I_{0 \to 2}) \cong [6,8] \oplus [6,8] \oplus [6,8].
\]
Finally, we obtain:
\[
\Pru_1(M) \cong I/K(-1) \cong [5,7]  \oplus [5,7] \oplus [5,7].
\]

\bigskip

\subsection{Pruning distance}

The~$\varepsilon$-pruning naturally leads to the definition of the \emph{pruning distance}. This distance is a promising candidate for a well-behaved metric in the multi-parameter setting. The main goal of this paper is to take a step toward establishing stability of this distance in the upset-decomposable modules. In addition to its stability properties, there are indications that the pruning distance may admit polynomial time computation, raising the prospect that it combines theoretical robustness with algorithmic tractability.

\begin{definition}\label{def_infty_refinement}
    
An \emph{$\infty$-refinement} of~$M\cong \bigoplus_{i \in I} M_i$ (where each~$M_i$ is indecomposable) is a module isomorphic to a module~$\bigoplus_{i \in I} N_i$ such that~$N_i$ is a subquotient of~$M_i$ for each~$i \in I$.

\end{definition}

The definition of~$\varepsilon$-refinement is described in \cite{bjerkevik_stabilizing} but will not be needed for the result of this paper.

\begin{definition}
\label{def_pruning_distance}
The \emph{pruning distance}~$d_P$ is defined by
\begin{align*}
d_P(M,N) = \inf\{&\varepsilon\geq 0\mid \forall \delta\geq 0,\\ &\Pru_{\varepsilon+\delta}(M) \text{ is an } \infty\text{-refinement of } \Pru_\delta(N) \text{ and }\\ &\Pru_{\varepsilon+\delta}(N) \text{ is an } \infty\text{-refinement of } \Pru_\delta(M)\}
\end{align*}
for any pfd modules~$M$ and~$N$.
\end{definition}

\begin{proposition}\label{propo_pruning_distance}
The pruning distance is a distance.
\end{proposition}

The conjecture below states that this distance is stable, and even Lipschitz equivalent to the interleaving distance if we restrict ourselves to modules of a bounded pointwise dimension.

For a module~$M$, let 
$
\operatorname{supdim} M = \sup_{p \in \mathbb{P}} \dim M_p \in \{0,1,\ldots,\infty\}.
$

\begin{conjecture}
\label{conjecture}
Let~$M$ and~$N$ be pfd modules with~$r = \operatorname{supdim} M<\infty$.
Then
\[
d_P(M,N)\leq d_I(M,N)\leq 2r\,d_P(M,N).
\]
\end{conjecture}

While the left-hand inequality of the conjecture was verified in~\cite{bjerkevik_stabilizing} in the case of~$\delta = 0$, its extension to arbitrary positive parameters remains open.

In~\cite{bjerkevik_stabilizing} many properties of the pruning were proved, notably~$\operatorname{Pru}_{\alpha}(M) \sim_{2r\alpha} M$ and the paper includes the following heuristic reasoning toward proving the right-hand inequality. Since~$d_P(M,N) \le \alpha$, the module~$\Pru_\alpha(M)$ is an~$\infty$-refinement of~$N$ and
$\Pru_\alpha(N)$ is an~$\infty$-refinement of~$M$. Hence, there exists a module~$P_1 := \Pru_\alpha(M)$ such that~$M \sim_{2r\alpha} P_1$ and~$P_1$ is an~$\infty$-refinement of~$N$. Similarly, there exists a module~$P_2 := \Pru_\alpha(N)$ such that~$N \sim_{2r\alpha} P_2$ and~$P_2$ is an~$\infty$-refinement of~$M$. This reasoning hopes to combine these statements and construct a~$2r\alpha$-interleaving between~$M$ and~$N$.

However, this intuition cannot be made precise in full generality and we show that using the following example. 

\begin{example}

We present in Figure \ref{fig_CI} a counter-example to the heuristic discussed above.
The construction is carried out within the class of upset-decomposable
modules, which are discussed in Section~\ref{sec:pruning-upset}.
These modules of the counter-example are obtained via suitable formalism of
constrained invertibility problems, or CI
problems for short, introduced in \cite{bjerkevik_botnan}. We do not recall the definition of CI problems here, referring instead to \cite{bjerkevik_botnan} and \cite{bjerkevik_Botnan_kerber} for the necessary background.
\begin{figure}
\centering
\begin{tikzpicture}[scale=1]
\node at (-3,.5){$\begin{bmatrix}
*&0\\
0&*
\end{bmatrix}$
$\begin{bmatrix}
0&*\\
*&0
\end{bmatrix}$};
\draw[color=black,fill=black] (0,0) circle (.1);
\node at (-.4,0){$M_2$};
\draw[color=black,fill=black] (0,1) circle (.1);
\node at (-.4,1){$M_1$};
\draw[color=black,fill=black] (2,0) circle (.1);
\node at (2.4,0){$N_2$};
\draw[color=black,fill=black] (2,1) circle (.1);
\node at (2.4,1){$N_1$};
\draw[thick, ->, shorten <=.3cm, shorten >=.3cm] (0,1) to (2,1);
\draw[thick, ->, shorten <=.3cm, shorten >=.3cm] (2,0) to (0,1);
\draw[thick, ->, shorten <=.3cm, shorten >=.3cm] (0,0) to (2,0);
\draw[thick, ->, shorten <=.3cm, shorten >=.3cm] (2,1) to (0,0);
\begin{scope}[xshift=4cm, yshift=2cm, scale=.15]
\draw[thick] (0,5) to (0,0) to (10,0) to (10,-6) to (19,-6) to (19,-13) to (25,-13) to (25,-23) to (29,-23);
\draw[thick,color=red] (1,5) to (1,1) to (11,1) to (11,-5) to (16,-5) to (16,-16) to (26,-16) to (26,-22) to (29,-22);
\draw[thick, dashed] (2,5) to (2,2) to (8,2) to (8,-8) to (17,-8) to (17,-15) to (27,-15) to (27,-21) to (29,-21);
\draw[thick,dashed,color=red] (3,5) to (3,3) to (9,3) to (9,-7) to (18,-7) to (18,-14) to (24,-14) to (24,-24) to (29,-24);
\draw[color=black,fill=black] (0,0) circle (.5);
\node[left] at (0,0){$p_0$};
\draw (0,0) to (5,5);
\draw[color=black,fill=black] (8,-8) circle (.5);
\node[left] at (8,-8){$p_1$};
\draw (8,-8) to (13,-3);
\draw[color=black,fill=black] (16,-16) circle (.5);
\node[left] at (16,-16){$p_2$};
\draw (16,-16) to (21,-11);
\draw[color=black,fill=black] (24,-24) circle (.5);
\node[left] at (24,-24){$p_3$};
\draw (24,-24) to (29,-19);
\end{scope}
\end{tikzpicture}
\caption{A CI problem on the left and the graph associated to it in the middle. On the right is an example of two upset-decomposable modules~$M = M_1 \oplus M_2~$and~$N = N_1 \oplus N_2$ associated to the CI problem, constructed using the proof of {\cite[Lemma 6.12]{bjerkevik_stabilizing}} (with~$C = 4$). The solid black, dashed black, solid red and dashed red curves are the boundaries of the upsets~$M_1$,~$M_2$,~$N_1$ and~$N_2$, respectively. For every~$i$ in~$\{0, 1, 2, 3\}$,  ~$p_i = (8i,-8i) \in \R^2$ and each diagonal step is adding~$(1,1)$ to the point before it. This construction gives us that~$d_I(M,N)=3$ using the same lemma.
}

\label{fig_CI}
\end{figure}
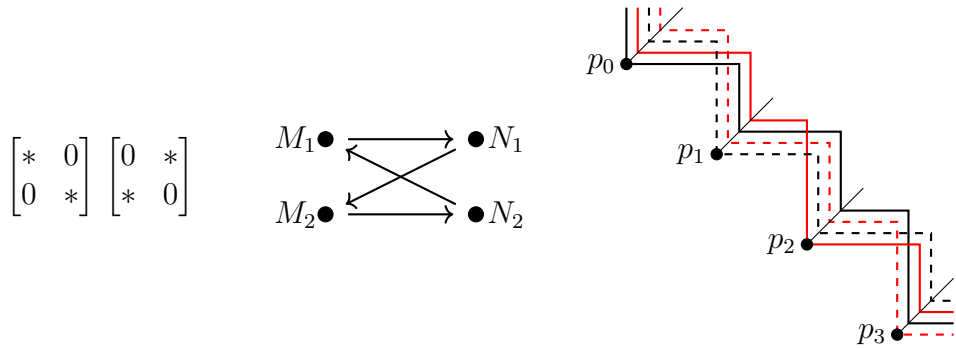

Let~$M = M_1 \oplus M_2~$and~$N = N_1 \oplus N_2$ be two upset-decomposable modules associated to Figure \ref{fig_CI}. We have~$M_1 \subseteq N_1(1)$ and~$M_2 \subseteq N_2(1)$ so~$M(-1)$ is an~$\infty$-refinement of~$N$ (and~$M(-1) \sim_1 M$). We also have~$N_1 \subseteq M_2(1)$ and~$N_2 \subseteq M_1(1)$ so~$N(-1)$ is an~$\infty$-refinement of~$M$ (and~$N(-1) \sim_1 N$), but~$d_I(M,N)=3$. Hence, despite satisfying the conditions stated above with~$P_1 = M(-1)$ and~$P_2 = N(-1)$, ~$M$ and~$N$ are not 1-interleaved, showing that the heuristic cannot be made precise in general.
\end{example}

In this paper, we explicitly compute the pruning for upset-decomposable modules in order to prove the inequality for this class of modules.

\clearpage
\section{Pruning of upset-decomposable modules}\label{sec:pruning-upset}
\markboth{3. Pruning of upset-decomposable modules}{}

\setcounter{theorem}{0} 
\renewcommand{\thetheorem}{3.\arabic{theorem}}

\subsection{Upset-decomposable modules}

In this section we recall the definition of upset-decomposable modules, the class of modules in which the main result of the next section will be proved. We also state and prove several standard lemmas that will be used repeatedly in the sequel. These results are well known in the literature. Precise statements and proofs of closely related results can be found as Lemmas~\ref{lem_upset_map1}, \ref{lem_upset_map2} and \ref{lem_upset_decomposable_map} in \cite{bjerkevik_Botnan_kerber}.

\begin{definition}\label{def_upset}

An \emph{upset interval} is a set U~$\subset \mathbb{R}^d$ such that~$\forall p \in U, \forall q \geq p$ is in U. \\
An \emph{upset module} M associated to U is a module defined by~$M_p=k$ if p~$in$ U and 0 else.~$M_{p \to q}=\mathrm{Id}_k, \forall p \in U$. And if p~$\notin$ U,~$M_p=0$ so automatically~$M_{p \to q} = 0$.\\
An \emph{upset-decomposable module} is a module isomorphic to a direct sum of upset modules.

\end{definition}

By abuse of notation, for upset modules~$M$ and~$N$, we write~$M \subseteq N$ to mean that~$\operatorname{supp}(M) \subseteq \operatorname{supp}(N)$,
which is equivalent to the existence of a monomorphism~$M \hookrightarrow N$.
If~$M = \bigoplus_{l=1}^r M_l$ and~$N = \bigoplus_{l=1}^r N_l$
are upset-decomposable, we write~$M \subseteq N$ to mean that
$M_l \subseteq N_l$ for all~$l = 1,\dots,r$.

\begin{lemma}\label{lem_shift_injection}
Let~$M$ be an upset module. Then for all~$\varepsilon \le \alpha$ we have
\[
M(\varepsilon) \subseteq M(\alpha),
\]
i.e., the shift morphism~$M_{\varepsilon \to \alpha} : M(\varepsilon) \to M(\alpha)$ is injective.
Therefore, if~$M=\bigoplus_{l=1}^r M_l$ is an upset-decomposable modules \[
\pi_l(M(\varepsilon)) \subseteq \pi_l(M(\alpha)), \forall l \in \{1,2,...,r\},
\]
with~$\pi_l$ denoting the canonical projection from~$M$ onto~$M_l$ in the direct sum.

\end{lemma}

\begin{proof}
For any~$p \in \Pb$, we have~$\Pb_\varepsilon(p) \le \Pb_\alpha(p)$ and since~$M$ is an upset module~$
(M_{\varepsilon \to \alpha})_p 
= M_{\Pb_\varepsilon(p) \to \Pb_\alpha(p)}
$ is injective for every~$p$.
Thus~$M_{\varepsilon \to \alpha}$ is an injection, so
$M(\varepsilon) \subseteq M(\alpha)$.
\end{proof}

\begin{lemma}\label{lem_upset_submodule}
Let~$M$ be an upset module. Then every submodule~$N$ of~$M$ is an upset module.
\end{lemma}

\begin{proof}
Let~$N$ be a nonzero submodule of~$M$.  
Pick any~$p \in \operatorname{supp}(N)$ and any~$q \ge p$. Since~$N \subseteq M$ and~$M$ is upset-decomposable, we have~$p,q \in \operatorname{supp}(M)$. In particular,$
M_p = M_q = k$, so the inclusion
$0 \neq N_p \subseteq M_p = k$ forces~$N_p = k$. And since~$M_{p \to q} = \mathrm{id}_k$, the diagram induced by the
submodule~$N \subseteq M$ is
\[
\begin{tikzcd}[row sep=3em, column sep=4em]
N_p = k \arrow[r] \arrow[d,hook] & N_q \arrow[d,hook] \\
M_p = k \arrow[r, "\mathrm{id}_k"'] & M_q = k
\end{tikzcd}
\]
The bottom arrow is an isomorphism and the vertical maps are injections so we get that the top arrow is also an isomorphism. In particular~$N_q = k$ and~$N_{p \to q} =\mathrm{id}_k$.
So~$q \in \operatorname{supp}(N)$.  
Therefore~$N$ is an upset module.
\end{proof}

\begin{lemma}\label{lem_upset_subquotient}
Let~$M$ and~$N$ be a nonzero upset modules.
If~$N$ is a subquotient of~$M$ then~$N$ is isomorphic to a submodule of~$M$ so~$N \subseteq M.$
\end{lemma}

\begin{proof}
\[
N \cong M'/M''
\]
with~$M''\subseteq M'\subseteq M$.  Since~$M$ is an upset module, every submodule of~$M$ is an upset module by Lemma~\ref{lem_upset_submodule}, so~$M'$ and~$M''$ are upset modules.  Because~$N$ is nonzero, there exists
$p_1\in\operatorname{supp}(M')\setminus\operatorname{supp}(M'')$. Assume for contradiction that~$M''\neq 0$. Pick~$p_2\in\operatorname{supp}(M'')$.Since we are working in~$\mathbb{R}^d$, we can pick~$p_3$ such that~$p_3\ge p_1$ and~$p_3\ge p_2$. We have
\[
p_3\in\operatorname{supp}(M')\quad\text{and}\quad p_3\in\operatorname{supp}(M'').
\]
Hence~$N_{p_3} \cong (M'/M'')_{p_3}=0$. But~$p_3\ge p_1$ so~$p_3\in\operatorname{supp}(N)$, a contradiction. Therefore~$M''=0$. Thus~$N\cong M'$ with~$M'$ a submodule of~$M$.
\end{proof}

\begin{lemma}\label{lem_upset_map1} 
Let~$M$ and~$N$ be nonzero upset modules, and let
$f\colon M\to N$. Then there exists
$\lambda\in k$ such that for every~$p\in\operatorname{supp}(M)$,~$
f_p(1)=\lambda.
$
\end{lemma}

\begin{proof}
Pick~$p_1 \in \operatorname{supp}(M)$ and set~$\lambda = f_{p_1}(1)$.  
Let~$p_2 \in \operatorname{supp}(M)$ let us prove that~$\lambda = f_{p_2}(1).$
Choose~$p_3 \in \operatorname{supp}(M)$ such that~$p_1 \le p_3$ and~$p_2 \le p_3$.  
By Lemma~\ref{lem_shift_injection}, the shift maps in~$M$ and~$N$ from~$p_1$ to~$p_3$  
and from~$p_2$ to~$p_3$ are injections.  
Using commutative diagrams identical in shape to the one in Lemma~\ref{lem_upset_submodule},  
we obtain
\[
\lambda = f_{p_1}(1) = f_{p_3}(1) = f_{p_2}(1).
\]
\end{proof}

\begin{lemma}\label{lem_upset_map2}
Let~$M$ and~$N$ be upset modules.  
If~$M \not\subseteq N$, then the only possible morphism~$f\colon M \to N$ is the zero map.  
If~$M \subseteq N$, then every choice of~$\lambda \in k$ yields a morphism~$f\colon M \to N$.
\end{lemma}

\begin{proof}
First case:~$M \not\subseteq N$.  
Then there exists~$p_1 \in \operatorname{supp}(M)$ such that~$p_1 \notin \operatorname{supp}(N)$.  
Thus~$f_{p_1}\colon M_{p_1}\to N_{p_1}=0$ is forced to be the zero map.  
By Lemma~\ref{lem_upset_map1},~$f_p(1)$ must be constant on~$\operatorname{supp}(M)$, so  
$f_{p}(1)=0$ for all~$p$, hence~$f$ is the zero morphism.

Second case:~$M \subseteq N$.  
Let~$\lambda\in k$ and define~$f_p\colon M_p\to N_p$ by~$f_p(1)=\lambda$ for every  
$p\in\operatorname{supp}(M)$ (and~$f_p=0$ if~$M_p=0$).  
By Lemma~\ref{lem_upset_map1}, this is the only possible form of a nonzero morphism.  
\end{proof}

\begin{lemma}\label{lem_upset_decomposable_map}
Let~$M=\bigoplus_{l=1}^r M_l \text{ and } N=\bigoplus_{s=1}^r N_s$ be upset–decomposable modules.  
A map~$f\colon M \to N$ is a morphism if and only if each component
$
f_{l}^s\colon M_l \longrightarrow N_s
$
is a morphism of modules of upset modules,~$\forall l,s \in \{1,\dots,r\}$.  
Therefore any morphism~$f$ can be determined by an~$r\times r$ matrix with entries in~$k$.
\end{lemma}
\begin{proof}
    Let~$p \leq q$, consider the folowing diagram: 
    \[
\begin{tikzcd}[row sep=3em, column sep=4em]
M_p \arrow[r, " M_{p \to q}"] \arrow[d, "f_p"] & M_q \arrow[d, "f_q"] \\
N_p \arrow[r, " N_{p \to q}"] & N_q
\end{tikzcd}
\]
Since
$
M_p=\bigoplus_{l=1}^r (M_l)_p \text{ and }
N_q=\bigoplus_{s=1}^r (N_s)_q
$, the diagram commutes if and only if for all~$l$ and~$s$, the restrictions to~$(M_l)_p$ followed by the projection to~$(N_s)_q$  of the two compositions are the same.
Equivalently, for all~$l$ and~$s$ the following diagram must commute:
\[
\begin{tikzcd}[column sep=4em,row sep=3em]
(M_l)_p \arrow[r,"(M_l)_{p\to q}"] \arrow[d,"(f^s_l)_p"'] &
(M_l)_q \arrow[d,"(f^s_l)_q"] \\
(N_s)_p \arrow[r,"(N_s)_{p\to q}"'] &
(N_s)_q
\end{tikzcd}
\]
where~$(f^s_l)_p$ denotes the restriction of~$f_p$ to the summand~$M_l$ followed by projection to the summand~$N_s$.
But~$f=\{f_p\}$ defines a morphism if and only if the first diagram commutes for all~$p\le q$, and the restriction~$(f^s_l)_p$ defines a morphism~$M_l\to N_s$ if and only if the second diagram commutes for all~$p\le q$, thus we have proved the equivalence.
\end{proof}

Let us denote the matrix described in Lemma~\ref{lem_upset_decomposable_map} by~$A^f$, as an immediate result of Lemma~\ref{lem_upset_map2}
$
A^f_{l,s}=0\text{ whenever } M_l \not\subseteq N_s .
$
We can determine~$A^f_{l,s}$ (therefore~$f$) just by determining~$f_p$ for any p~$\in \bigcap_{l=1}^r \operatorname{supp}(M_l)$.

$M_{l_1} \bigcap M_{l_2}$ will denote the upset module with~$\operatorname{supp}(M_{l_1} \bigcap M_{l_2})$ =~$\operatorname{supp}(M_{l_l}) \bigcap \operatorname{supp}(M_{l_2}).$\\

The following lemma does not seem to appear explicitly in the existing literature and will be needed in the sequel.

\begin{lemma}\label{lem_upset_decomposable_inverse_image}
Let~$M=\bigoplus_{l=1}^r M_l \text{ and } N=\bigoplus_{s=1}^r N_s$ be upset–decomposable modules, let~$f:M \to N$ and~$N' = \bigoplus_{s=1}^r N'_s \subseteq N.$ Then
\[
\left( f_{|_{M_l}} \right)^{-1}(N')
  = M_l \bigcap \left( \bigcap_{s \in J_l^f} N'_s \right).
\] where~$J_{l}^f=\{s \in \{1,2,...,r\}: A_{l,s}^f \neq 0\} \subseteq \{1,2,...,r\}.$
Moreover,
$
f^{-1}(N') =\bigoplus_{l=1}^r \pi_l(f ^{-1}(N'))= \bigoplus_{l=1}^r \left( f_{|_{M_l}} \right)^{-1}(N').
$
\end{lemma}

\begin{proof}
    By Lemma~\ref{lem_upset_submodule},
$\left(f_{|_{M_l}}\right)^{-1}(N')$ is an upset module, since it is a submodule of the upset module~$M_l$.
Now let us find its support:\begin{itemize}[leftmargin=*]
 \item If~$p \notin \operatorname{supp}(M_l)$ then~$(M_l)_p = 0$ therefore~$(\left(f|_{M_l}\right)^{-1}(N'))_p$=0.
\item If~$p \in \operatorname{supp}(M_l)$ but~$p \notin \bigcap_{s \in J_l^f} \operatorname{supp}(N'_s)$, then there exists~$s \in J_l^f$ such that~$p \notin \operatorname{supp}(N'_s)$.  
This implies that~$(f_l^s)_p(1)) = A_{l,s}^f \neq 0$ and~$(N'_s)_p = 0$, it follows that~$(f{|_{M_l}})_p(1) \notin N'_p$.  
Because~$(M_l)_p = k$, we get~$(\operatorname{im}(f{|_{M_l}}))_p \cap N'_p = 0$, hence~$((f{|_{M_l}})^{-1}(N'))_p = 0$.
\item If~$p \in \operatorname{supp}(M_l) \bigcap \left( \cap_{s \in J_l^f} \operatorname{supp}(N'_s) \right)
$ (if any of the~$N'_s=0$ nothing to prove).
Let~$x \in (M_l)_p,$ if~$s \notin J_l^f, f_l^s(x))=0
$ and if~$s \in J_l^f$ we have that~$(N'_s)_p=(N_s)_p$ so~$(\operatorname{im}(f{|_{M_l}}) \subset N'_p$ and~$\left((f_{|_{M_l}})_p\right)^{-1}(N')=(M_l)_p=k.$
\end{itemize}
Therefore
$
\left(f|_{M_l}\right)^{-1}(N') \;=\; M_l \bigcap \left(\bigcap_{s\in J_l^f} N'_s\right).
$ \\

\end{proof}

\subsection{Pruning of upset-decomposable modules}

In this section we prove the main result of the paper, the right-hand inequality in Conjecture~\ref{conjecture} for upset-decomposable modules. The proof proceeds by an explicit computation of the pruning of these modules.
We first translate the problem into a combinatorial setting by associating a directed graph to an upset an upset-decomposable module, which allows us to compute the pruning in a systematic way.

Let~$\alpha \ge 0$ be a fixed constant.
To compute the~$\alpha$-pruning of an upset-decomposable module
$
M=\bigoplus_{l=1}^r M_l,
$
we first associate to~$M$ an directed graph~$G(M,\alpha)$.
The vertices of~$G(M,\alpha)$ are~$\{M_1,\dots,M_r\}$, and there is an edge
$
M_{l_1} \to M_{l_2}\text{, if } M_{l_1} \subseteq M_{l_2}(2\alpha).
$ In what follows a path will mean a directed path.\\
For each~$l_1 \in \{1,\dots,r\}$, define the sets of reachable indices from~$l_1$:
\begin{itemize}
\item~$
JM_{l_1}^0 = \{l_1\},
$

\item~$
JM_{l_1}^u =
\bigl\{\,l \in \{1,\dots,r\}
\;\big|\;
\exists\text{ path of length } \le u
\text{ from }M_{l_1}\text{ to }M_l
\bigr\},\forall u \in \{1,\dots,r-1\},
$
\item~$JM_{l_1} =\bigl\{ l \in \{1, \dots, r\} \;\big|\;\exists \text{ path from } M_{l_1} \text{ to } M_l \bigr\}.
$

\end{itemize}
We write~$M_{l_1} \to^* M_l$ to indicate the existence of such a path.
We notice that 
$
JM_{l_1} = JM_{l_1}^{r-1},
$
since all paths without repeated vertices have length at most~$r-1$.\\
For~$S \subseteq \{1,\dots,r\}$, we define
$
JM_S \coloneqq \bigcup_{l \in S} JM_l .
$

    \begin{theorem}\label{theorem_1}
Let~$\alpha \geq 0$ and let~$M = \bigoplus_{l=1}^r M_l$ be an upset-decomposable module. Then
\[
\pi_{l_1}(\operatorname{Pru}_\alpha(M)) = \bigcap_{l \in JM_{l_1}} M_l(-\alpha),
\]
where~$JM_{l_1}$ is as defined above.
\end{theorem}

\begin{proof}
To compute the~$\alpha$-pruning, we use the construction in the proof of Lemma~\ref{lem_pruning_construction}, i.e., we compute the modules~$K$ and~$I$.  
Since~$M$ is upset-decomposable, we get that~$I$ is also upset-decomposable by Lemma~\ref{lem_upset_submodule}. \\ 
By Lemma~\ref{lem_shift_injection}, the shift morphism of~$I$ is injective. Therefore,
\[
K_1 = \sum_{j \in \Gamma} (I_{0 \to 2\alpha})^{-1}(f_j(0)) = 0=K_0,
\]
with~$\{f_j\}_{j \in \Gamma}$ being the set of morphisms from~$M$ to~$M(2\alpha)$. So~$K = 0$.
Hence, we have 
\[
\operatorname{Pru}_\alpha(M) = I(-\alpha).
\]
Next, we compute~$I$. For that we need to compute~$I_i$ for every i.
We claim that: \[
\pi_{l_1}(I_i)= \bigcap_{l \in JM_{l_1}^i} M_l.
\]
For~$i$=0, the claims hold.
Suppose by induction that it is true for~$i-1$ let us prove it for~$i$.
\[\pi_{l_1}(I_{i}) = \bigcap_{j \in \Gamma} \pi_{l_1}(f_j^{-1}(M_{0 \to 2\alpha}(I_{i-1}))).\]
By Lemma~\ref{lem_shift_injection},~$\pi_l(M_{0 \to 2\alpha}(I_{i-1})) \cong \pi_l(I_{i-1}),\forall l \in \{1,2,...,r\}.$ 
Using Lemma~\ref{lem_upset_decomposable_inverse_image}
 \[\pi_{l_1}(f_j^{-1}(I_{i-1}))=\pi_{l_1}(I_{i-1}) \bigcap
\left(
\bigcap_{v \in J_{l_1}^{f_j}} \pi_l(I_{i-1})
\right)\]
 Since~$\{f_j\}_{j \in \Gamma}$ is the set of all possible morphisms from~$M$ to~$M(2\alpha)$ then~$\bigcup_{j \in \Gamma} J_{l_1}^{f_j} = JM_{l_1}^1$ and we have that~$l_1 \in JM_{l_1}^1$ therefore
 \[
\pi_{l_1}(I_{i}) = \bigcap_{j \in \Gamma} \pi_{l_1}(f_j^{-1}(I_{i-1})) = \bigcap_{v \in JM_{l_1}^1} \pi_v(I_{i-1})
 \]
 So using the induction hypothesis we get that:
 \[
\pi_{l_1}(I_{i}) = \bigcap_{v \in JM_{l_1}^1} \left(\bigcap_{l \in JM_v^{i-1}} M_l \right)= \bigcap_{l \in JM_{l_1}^i} M_l.
 \]
The last equality came from the fact that~$l$ being at distance less than~$i-1$ (pathwise) of from a point~$v$ that is at distance less than 1 from~$l_1$ is the same as~$l$ being at a distance less that~$i$ from~$l_1$.
So \[\pi_{l_1}(I)=\bigcap_{i=0}^\infty \pi_{l_1}(I_i)=\bigcap_{l \in \bigcup_{i=0}^{r-1} JM_{l_1}^i} M_l=\bigcap_{l \in JM_{l_1}} M_l\]
So we get:
\[
\pi_{l_1}(\operatorname{Pru}_\alpha(M)) =\pi_{l_1}(I)(-\alpha)=  \bigcap_{l \in JM_{l_1}} M_l(-\alpha).
\]
\end{proof}

Let~$G$ be a directed graph with~$|V(G)| < \infty$. Let~$A,B \subseteq V(G).$
For~$a,b \in V(G)$, we let~$l_G(a,b)$ be the length of the shortest path from~$a$ to~$b$,
with~$l_G(a,a)=0$ and~$l_G(a,b)=\infty$ if no path from~$a$ to~$b$ exists. Set
\[
d_G(A,B):=\max_{b\in B}\min_{a\in A} l_G(a,b).
\]
In other words, it is the shortest length of path needed to reach every vertex in B starting from some vertex in A.
Clearly,~$d_G(A,B) \neq d_G(B,A)$ in general.

\begin{lemma}\label{lem_graph_1}
Let~$A\subseteq V(G)$ and set
\[
B:=\{\,b\in V(G)\mid d_G(A,\{b\})<\infty\,\},
\]
i.e.\ the set of vertices reachable from~$A$. Then
\[
d_G(A,B)\le |B|-|A|,
\]
where~$|\cdot|$ denotes cardinality.
\end{lemma}

\begin{proof}
Assume for contradiction that~$d_G(A,B)=c>|B|-|A|$, so there exists
$b_1\in B$ with~$\min_{a\in A} l_G(a,b_1)=c$. Then for every~$a\in A$ we have
$l_G(a,b_1)\ge c$, and there is~$a_1\in A$ such that~$l_G(a_1,b_1)=c$. Let
\[
P: a_1=v_0\to v_1\to\cdots\to v_{c-1}\to v_c=b_1
\]
be a path of length~$c$. The~$c-1$ intermediate vertices~$v_1,\dots,v_{c-1}$
are pairwise distinct (otherwise the path could be shortened).
In addition~$b_1$ cannot be an intermediate vertex or else the path could again be shortened.
None of the intermediate
vertices lies in~$A$, since if~$v_j\in A$ then~$l_G(v_j,b_1)\le c-j<c$,
so we will get an element in~$A$ with
a path to~$b_1$  strictly shorter than c.\\
Each intermediate vertex is clearly in~$B$. Thus the
intermediate vertices form a subset of~$B\setminus (A\cup \{b_1\})$.
If~$b_1\in A$ then~$c=0$, contradicting~$c>|B|-|A|$.
If~$b_1 \notin A$ we get that~$|B\setminus (A\cup \{b_1\})|=|B|-|A|-1$. Therefore
\[
c-1 \le |B|-|A|-1,
\]
so~$c\le |B|-|A|$, contradiction. Hence~$d_G(A,B)\le |B|-|A|$.
\end{proof}

\begin{lemma}\label{lem_graph_2}
Let~$M=\bigoplus_{l=1}^r M_l$ be an upset-decomposable module and let~$G=G(M,\alpha)$ be the directed graph associated to~$M$ for some~$\alpha\ge0$ as defined above.
Then \[l_G(M_{l_1},M_{l_2}) \leq c   \quad\Longrightarrow\quad M_{l_1} \subseteq M_{l_2}(2c\alpha). \]

\end{lemma}

\begin{proof}
Assume~$l_G(M_{l_1},M_{l_2})\le c$. Then there exists a path
\[
P:\; M_{l_1}=M_{l'_0}\to M_{l'_1}\to\cdots\to M_{l'_{c'-1}}\to M_{l'_{c'}}=M_{l_2},
\]
with~$c' \le c.$ 
Then by the construction of~$G$ each arrow gives a containment
shifted by~$2\alpha$, hence \[M_{l_1} \subseteq M_{l'_1}(2\alpha) \subseteq M_{l'_2}(4\alpha) \subseteq\cdots\subseteq M_{l'_{c'-1}}(2(c'-1)\alpha) \subseteq M_{l_2}(2c'\alpha).\]
Since~$c' \leq c$ we have~$M_{l_2}(2c'\alpha)\subseteq M_{l_2}(2c\alpha)$ by Lemma~\ref{lem_shift_injection}, which proves the claim.

\end{proof}

\begin{lemma}\label{lem_graph_3}
 Let~$N$ be an upset module, let~$M=\bigoplus_{l=1}^r M_l$ be an upset-decomposable module and let~$G=G(M,\alpha)$ be the directed graph associated to~$M$ for some~$\alpha\ge0$.
Let~$A,B \subseteq V(G)$ as in Lemma~\ref{lem_graph_1}, then:
\[
\forall M_a \in A,\; N \subseteq M_a \quad\Longrightarrow\quad \forall M_b \in B, \; N \subseteq M_b(2(|B|-|A|)\alpha).
\]
\end{lemma}

\begin{proof}
Let~$M_b \in B.$ Then  
\[d_G(A,\{M_b\}) \leq d_G(A,B) \leq |B|-|A| \text{ by Lemma~\ref{lem_graph_1}.}\]
\begin{align*}
&\implies \min_{M_a \in A}l_G(M_a,M_b) \leq |B|-|A|\\
&\implies \exists M_{a_1} \in A \mid l_G(M_{a_1},M_b) \leq |B|-|A|\\
&\implies \exists M_{a_1} \in A \mid M_{a_1} \subseteq M_b(2(|B|-|A|)\alpha) \text{ by Lemma~\ref{lem_graph_2}}\\
&\implies N \subseteq M_{a_1} \subseteq M_b(2(|B|-|A|)\alpha).\\
\end{align*}
Thus the claim.
    
\end{proof}

\begin{lemma}\label{lem_infinit_refinement}
    Let~$M=\bigoplus_{l=1}^r M_l$ and~$N=\bigoplus_{l=1}^r N_l$ be upset-decomposable modules. If~$N$ is an~$\infty$-refinement of~$M$. There exists a permutation~$\sigma$ of~$\{1,\dots,r\}$ such that
\[N_{\sigma(l)} \subseteq M_l,
\forall l \in \{1,\dots,r\}.
\]
\end{lemma}

\begin{proof}
Let~$N$ be an~$\infty$-refinement of~$M$. By definition there is a decomposition
\[
N \cong \bigoplus_{l=1}^r N'_l
\]
such that each~$N'_l$ is a subquotient~$M_l$. 
By Lemma~\ref{lem_upset_subquotient} each~$N'_l$ is an upset module so it is indecomposable and~$N'_l \subseteq M_l$. Since~$N'_l$ is indecomposable and nonzero, by Theorem~\ref{th_barcode} there exists a permutation~$\sigma$ of~$\{1,\dots,r\}$ such that~$N'_l \cong N_{\sigma(l)}.$ Thus~$N_{\sigma(l)} \subseteq M_l$.
\end{proof}

\begin{lemma}\label{lem_pruning_inclusion}
    Let~$M=\bigoplus_{l=1}^r M_l$ be upset-decomposable and let~$G=G(M,\alpha)$ be the directed graph associated to~$M$ for some~$\alpha\ge0$. If~$l_2 \in JM_{l_1}$ then~$\pi_{l_1}(\operatorname{Pru}_\alpha(M)) \subseteq \pi_{l_2}(\operatorname{Pru}_\alpha(M))$
    
\end{lemma}

   \begin{proof}
       By transitivity of~$\to^*$ we get~$JM_{l_2} \subseteq JM_{l_1}$ so \[\pi_{l_1}(\operatorname{Pru}_\alpha(M)) = \bigcap_{l \in JM_{l_1}} M_l(-\alpha) \subseteq \bigcap_{l \in JM_{l_2}} M_l(-\alpha)=\pi_{l_2}(\operatorname{Pru}_\alpha(M)),\] which proves the lemma.
   \end{proof}

   We now state and prove a main result of this paper, which proves the upper bound in Conjecture~\ref{conjecture} for the class of upset-decomposable modules. We improve the constant from~$2r$ to~$2r-1$ and strengthen the statement by replacing the interleaving distance~$d_I$ with the bottleneck distance~$d_B$; since~$d_I \leq d_B$, the conjectured bound for~$d_I$ follows immediately.
   
   We also show that this bound is optimal for both~$d_I$ and~$d_B$ in Example~\ref{ex_2}.
   
   \begin{theorem}\label{theorem_main}
       Let~$M$ and~$N$ be pfd upset-decomposable modules with ~$r$ =~$\operatorname{supdim}$~$M < \infty$ then
       \[d_B(M,N) \leq (2r-1)\,d_P(M,N).\]
   \end{theorem}
   \begin{proof}
   If~$\operatorname{supdim}$~$N \neq r$ then both distances~$= \infty$ so nothing to prove.\\
   Suppose that~$\operatorname{supdim}$~$N = r$ i.e~$M=\bigoplus_{l=1}^r M_l$ and~$N=\bigoplus_{l=1}^r N_l.$
   Let~$\alpha \geq 0$, such that~$\operatorname{Pru}_{\alpha}(M)$ is an~$\infty$-refinement of~$N$ and~$\operatorname{Pru}_{\alpha}(N)$ is an~$\infty$-refinement of~$M$, our goal is to prove that there exists an~$(2r-1)\alpha$-matching between~$M$ and~$N$. By Lemma~\ref{lem_infinit_refinement} there exists permutations~$\sigma_1,\sigma_2$ of~$\{1,\dots,r\}$ such that\[
   \pi_{\sigma_1(l)}(\operatorname{Pru}_{\alpha}(M)) \subseteq N_l\quad\text{and}\quad\pi_{\sigma_2(l)}(\operatorname{Pru}_{\alpha}(N)) \subseteq M_l,\quad\forall\,l \in \{1,\dots,r\}.\]Define~$N' = \bigoplus_{l=1}^r N'_l$ where~$N'_l = N_{\sigma_2(l)}$, then~$N' \cong N$.  Set~$\sigma = \sigma_2^{-1} \circ \sigma_1^{-1}$, we get that
   \begin{equation}
   \pi_l(\operatorname{Pru}_{\alpha}(M)) \subseteq N'_{\sigma(l)}\quad\text{and}\quad\pi_{l}(\operatorname{Pru}_{\alpha}(N')) \subseteq M_l,\quad\forall\,l \in \{1,\dots,r\}.
   \tag{1}
\label{eq:1}
\end{equation}
Since~$N' \cong N$ having matching between~$N$ and~$M$ is equivalent to having a  matching between~$M$ and~$N'$.
From now on, to simplify the notation,~$N$ refers to~$N'$.\\
The matching will identify~$M_l$ with~$N_l$ for all~$l \in \{1,\dots,r\}$. 
Therefore it suffices to show that
\[
M_l \subseteq N_l((2r-1)\alpha)\quad\text{and}\quad
N_l \subseteq M_l((2r-1)\alpha),\quad\forall\,l \in \{1,\dots,r\}.
\]
Fix~$l_1 \in \{1,\dots,r\}$.  
In the directed graph~$G(N,\alpha)$, set
\[
A = \{N_{l_1}\},\qquad B = \{N_l : l \in JN_{l_1}\}.
\]
Lemma~\ref{lem_graph_3} gives
\[
N_{l_1} \subseteq N_l((2|JN_{l_1}| - 2)\alpha)\quad \forall\,l \in JN_{l_1}
\]
hence
\[N_{l_1} \subseteq \bigcap_{l\in JN_{l_1}} N_l((2|JN_{l_1}|-2)\alpha)\]
So by Theorem~\ref{theorem_1} and \eqref{eq:1} we get:
\[N_{l_1} \subseteq \pi_{l_1}(\operatorname{Pru}_\alpha(N))((2|JN_{l_1}|-1)\alpha)
\subseteq M_{l_1}((2|JN_{l_1}|-1)\alpha).
\]
Since~$|JN_{l_1}|\le r$ we obtain
\[
N_{l_1} \subseteq M_{l_1}((2r-1)\alpha).
\]
We use a similar argument in~$G(M,\alpha)$.  
Let~$J_1 = JM_{l_1}$.  
We get that,
\begin{equation}
\forall l \in J_1, \quad M_{l_1} \subseteq M_l((2|J_1|-2)\alpha) 
 \tag{2}
\label{eq:2}
\end{equation}
and
\[
M_{l_1} \subseteq \pi_{l_1}(\operatorname{Pru}_\alpha(M))((2|J_1|-1)\alpha) 
\]
by Lemma~\ref{lem_pruning_inclusion} and~\eqref{eq:1} we get 
\[
M_{l_1} \subseteq \pi_l(\operatorname{Pru}_\alpha(M))((2|J_1|-1)\alpha)\subseteq N_{\sigma(l)}((2|J_1|-1)\alpha) \quad \forall \, l \in J_1 
\]
Equivalently,
\[
M_{l_1}(-(2|J_1|-1)\alpha) \subseteq N_l
\quad \forall \, l \in \sigma(J_1).
\]
Continue now in~$G(N,\alpha)$ we define~$J'_1=JN_{\sigma(J_1)}$, we use Lemma~\ref{lem_graph_3} with 
\[
A = \{N_l : l \in \sigma(J_1)\},\qquad
 B=\{N_l \mid l \in J'_1\}
\]
we get
\[
M_{l_1}(-(2|J_1|-1)\alpha) \subseteq
N_l(2(|J'_1|-|J_1|)\alpha)
\quad \forall\,l \in J'_1,
\]
so
\[
M_{l_1}(-(2|J'_1|-1)\alpha) \subseteq N_l,
\quad \forall\,l \in \sigma(J'_1).
\]
If~$l_1 \in J'_1$ we conclude
$M_{l_1}(-(2|J'_1|-1)\alpha)\subseteq N_{l_1}$
and we are done, since~$|J'_1| \leq r$.  
Otherwise we carry on our computation.
    By the construction of~$J'_1$ we have for all~$l \in J'_1, JN_l \subseteq J'_1$. Thus \\\[\forall l \in J'_1, \quad M_{l_1}(-(2|J'_1|-1)\alpha) \subseteq \bigcap_{s \in JN_{l}} N_s. \]
   Therefore again using Theorem~\ref{theorem_1} and \eqref{eq:1} we get that
\begin{equation}
  \forall l \in J'_1, \quad M_{l_1}(-2|J'_1|\alpha) \subseteq \pi_{l}(\operatorname{Pru}_\alpha(N)) \subseteq M_l. 
\tag{3}
\label{eq:3}
\end{equation}
Since~$(2|J_1|-2) \leq 2|J_1| \leq 2|J'_1|$ we combine the equations \eqref{eq:2} and \eqref{eq:3} and we get that
\[\forall l \in J'_1 \cup J_1, \quad M_{l_1}(-2|J'_1|\alpha) \subseteq M_l. \]
As we assumed before~$l_1 \notin J'_1$ and we know that~$l_1 \in J_1$ so~$|J'_1 \cup J_1|-1 \geq |J'_1|$, we get that
\[\forall l \in J'_1 \cup J_1, \quad M_{l_1}(-(2|J'_1 \cup J_1|-2)\alpha) \subseteq M_l. \]
We define~$J_2=JM_{J'_1 \cup J_1}$, we use again Lemma~\ref{lem_graph_3} with~$A=\{M_l \mid l \in J'_1 \cup J_1\}$ and~$B = \{M_l \mid l\in J_2 \}$ so
\[\forall l \in J_2, \quad M_{l_1} (-(2|J'_1 \cup J_1|-2)\alpha) \subseteq M_l(2(|J_2|-|J'_1 \cup J_1|)\alpha).\] thus
\[\forall l \in J_2, \quad M_{l_1}(-(2|J_2|-2)\alpha) \subseteq M_l.\]
We define the induction~$J'_n=JM_{\sigma(J_n)}$ and whenever~$l_1 \notin J'_{n}$, we pose
$J_{n+1}=JM_{J_{n}\cup J'_{n}}$. Since ~$l_1 \in J_n \text{ but } l_1 \notin J'_n$;
we use the fact that~$|J'_n \cup J_n| -1 \geq |J'_n|$ and we use repeatedly Lemma~\ref{lem_graph_3} to get
\[\forall l \in J_n, \quad M_{l_1}((2|J_n|-2)\alpha) \subseteq M_l.\]
and
\[\forall l \in J'_n, \quad M_{l_1}((2|J'_n|-1)\alpha) \subseteq N_l.\]
We carry on with our induction until we reach~$n=n_0$, with~$n_0$ being the first integer such that~$l_1 \in J'_{n_0}$. We claim that this induction is finite and~$n_0 \leq r$. We show by induction that for every~$n$,~$\sigma^n(l_1) \in J'_n$, this is true for~$n = 1$ because~$l_1 \in J_1$ and we have that~$\sigma(J_1) \subseteq J'_1$ so~$\sigma(l_1) \in J'_1$, and if~$\sigma^n(l_1) \in J'_n$, we have that~$J'_n \subseteq J_{n+1}$ and~$\sigma(J_{n+1}) \subseteq J'_{n+1}$ so~$\sigma^{n+1}(l_1) \in J'_{n+1}$. Since~$\sigma$ is a permutation of~$\{1,\dots,r\}$, there is~$n_0 \leq r, \, \sigma^{n_0}(l_1) = l_1$, which proves our claim.
So by that we obtain
\[
M_{l_1} \subseteq N_{l_1}((2|J'_{n_0}|-1)\alpha) \subseteq N_{l_1}((2r-1)\alpha).
\]
Since~$l_1$ was arbitrary,
\[
\forall\,l \in \{1,\dots,r\},\quad
M_l \subseteq N_l((2r-1)\alpha)
\quad\text{and}\quad
N_l \subseteq M_l((2r-1)\alpha).
\]
Thus we have constructed a~$(2r-1)\alpha$-matching and we conclude that
\[
d_B(M,N) \le (2r-1)\alpha = (2r-1)d_P(M,N).
\]
\end{proof}

\begin{example}
We will give an example that helps us visualize the proof. Let~$M=\bigoplus_{l=1}^r M_l$ and~$N=\bigoplus_{l=1}^r N_l$ be two upset-decomposable modules. The verticalal arrows of this diagram illustrate the~$G(M,\alpha)$ and~$G(N,\alpha)$, and the diagonal arrows illustrate~$\sigma$.

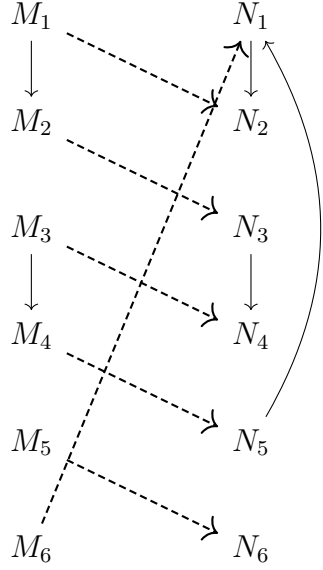
\[\begin{tikzcd}
	{M_1} && {N_1} \\
	{M_2} && {N_2} \\
	{M_3} && {N_3} \\
	{M_4} && {N_4} \\
	{M_5} && {N_5} \\
	{M_6} && {N_6}
	\arrow[from=1-1, to=2-1]
	\arrow[from=1-1, dashed, thick, to=2-3]
	\arrow[from=1-3, to=2-3]
	\arrow[from=2-1, dashed, thick, to=3-3]
	\arrow[from=3-1, to=4-1]
	\arrow[from=3-1, dashed, thick, to=4-3]
	\arrow[from=3-3, to=4-3]
	\arrow[from=4-1, dashed, thick, to=5-3]
	\arrow[from=5-1, dashed, thick, to=6-3]
	\arrow[from=6-1, dashed, thick, to=1-3]
	\arrow[bend right=30, from=5-3, to=1-3]
\end{tikzcd}\]
In this example with~$l_1=1,$ we have~$J_1 = JM_1 = \{1,2\}, \, J'_1 = JM_{\sigma(\{1,2\}} = JM_{\{2,3\}} =\{2,3,4\}, \, J_2 = JM_{\{1,2,3,4\}}=\{1,2,3,4\}, \, J'_2 = JM_{\sigma(\{1,2,3,4\}} = JM_{\{2,3,4,5\}} =\{1,2,3,4,5\}.$ Since~$1 \in J'_2$ we stop our induction. So we get that~$M_1 \subseteq N_1(9\alpha) \subseteq N_1(11\alpha).$
\end{example}

The following example illustrates that the upper bound~$(2r-1)d_P(M,N)$ in Theorem~\ref{theorem_main} can be attained for the bottleneck distance~$d_B$ and the interleaving distance~$d_I.$

\begin{example}\label{ex_2}
Let~$\alpha>0$ and~$r \in \mathbb{N^*}$. And let
$
M = \bigoplus_{i=1}^r M_i \text{ and } 
N = \bigoplus_{i=1}^r N_i,
$ be two~$d$-parameter upset-decomposable modules
with supports
\[
\operatorname{supp} (M_i) = \{(x_1,\dots,x_d) \mid (x_1,\dots,x_d) \ge (2i\alpha,\dots,2i\alpha)\}\]
\[
\operatorname{supp} (N_i) = \{(x_1,\dots,x_d) \mid (x_1,\dots,x_d) \ge \left( (2r+1)\alpha,\dots,(2r+1)\alpha \right)\}.
\]
Then one can check that 
\[
\operatorname{Pru}_{\alpha}(M) \cong N,
\]
so~$\operatorname{Pru}_{\alpha}(M)$ is an~$\infty$-refinement of~$N$. 
Moreover, for every~$\delta>0$, 
\[
\operatorname{Pru}_{\alpha+\delta} (M) \cong \operatorname{Pru}_{\delta} (N) \cong N(-\delta),
\] 
so we also have~$\operatorname{Pru}_{\alpha+\delta}(M)$ as an~$\infty$-refinement of~$N(-\delta)$.
We can also easily check that~$\forall \varepsilon < \alpha, \operatorname{Pru}_{\varepsilon}(M)$ is not an~$\infty$-refinement of~$N$.
Similarly,~$N$ is an~$\infty$-refinement of~$M$, and for every~$\delta>0$,~$\operatorname{Pru}_{\delta}(N)$ is an~$\infty$-refinement of~$\operatorname{Pru}_{\delta}(M)$.\\
From these observations we obtain that
\[
d_P(M,N) = \alpha.
\]
Furthermore, since all~$N_i$ have the same support, it is easy to check that
\[
d_B(M,N) = d_I(M,N) = (2r-1)\alpha,
\]
showing that the upper bound in Theorem~\ref{theorem_main} can indeed be attained ~$\forall \alpha>0$,~$\forall d \in \mathbb{N}^*$ and~$\forall r \in \mathbb{N}^*$. In particular, the constant~$2r-1$ is optimal, and the inequality cannot be strengthened neither for the bottleneck distance~$d_B$ nor for the interleaving distance~$d_I$.

\end{example}

We now state and prove another main result of this paper. This result provides a partial converse to Theorem~\ref{theorem_main}, providing a weaker bound than the left-hand inequality in Conjecture~\ref{conjecture}. Together, these two theorems establish that the pruning distance is Lipschitz equivalent to the bottleneck distance in the class of upset-decomposable modules.

\begin{theorem}\label{theorem_3}
       Let~$M$ and~$N$ be pfd upset-decomposable modules then
       \[d_P(M,N) \leq d_B(M,N).\]
   \end{theorem}
\begin{proof}
Let~$r = \operatorname{supdim}$~$M$, if~$\operatorname{supdim}$~$N \neq r$ then both distances~$= \infty$ so nothing to prove.\\
Suppose that~$\operatorname{supdim}$~$N = r$ i.e~$M=\bigoplus_{l=1}^r M_l$ and~$N=\bigoplus_{l=1}^r N_l.$ Let~$\alpha \geq 0$, such that there exists a matching between~$M$ and~$N$. This means that there exists a permutation~$\sigma$ of~$\{1,\dots,r\}$ such that 
\[
M_l \subseteq N_{\sigma(l)}(\alpha) \quad \text{and} \quad N_l \subseteq M_{\sigma^{-1}(l)}(\alpha), \quad \forall l \in \{1,\dots,r\}.
\]
Define~$N' = \bigoplus_{l=1}^r N'_l$ where~$N'_l = N_{\sigma(l)}$, then~$N' \cong N$. We get that
   \begin{equation}
M_l \subseteq N'_l(\alpha) \quad \text{and} \quad N'_l \subseteq M_l(\alpha), \quad \forall l \in \{1,\dots,r\}.
   \tag{4}
\label{eq:4}
\end{equation}
From now on, to simplify the notation,~$N$ refers to~$N'$.\\
We want to show that~$\forall \delta\geq 0, \Pru_{\alpha+\delta}(M) \text{ is an } \infty\text{-refinement of } \Pru_\delta(N)$ and~$\Pru_{\alpha+\delta}(N)$ is an~$\infty\text{-refinement of } \Pru_\delta(M).$ We prove only the first statement; the second follows by symmetry. By Lemma~\ref{lem_infinit_refinement} it is sufficient to show that 
\[
\forall \delta\geq 0, \, \forall l \in \{1,\dots,r\}, \quad \pi_{l}(\operatorname{Pru}_{\alpha+\delta}(M)) \subseteq \pi_{l}(\operatorname{Pru}_{\delta}(N)).
\]
Fix~$l_1 \in \{1,\dots,r\}$ and~$\delta \geq 0$. Let~$l \in JN_{l_1}$ in the directed graph~$G(N,\delta)$, then there exists a path
\[
P:\; N_{l_1}=N_{l'_0}\to N_{l'_1}\to\cdots\to N_{l'_{c-1}}\to N_{l'_{c}}=N_{l},
\]
for some integer~$c$.
By the construction of~$G(N,\delta)$ each arrow gives a containment shifted by~$2\delta$, hence
\[
N_{l'_0} \subseteq N_{l'_1}(2\delta), N_{l'_1} \subseteq N_{l'_2}(2\delta), \cdots , N_{l'_{c-1}} \subseteq N_{l'_{c}}(2\delta).
\]
By~\eqref{eq:4} we get that 
\[
M_{l'_i} \subseteq N_{l'_i}(\alpha) \subseteq N_{l'_{i+1}}(\alpha+2\delta) \subseteq M_{l'_{i+1}}(2\alpha+2\delta), \quad \forall i \in \{0,\dots,c-1\}.
\]
So we get that there exists a path 
\[
P':\; M_{l_1}=M_{l'_0}\to M_{l'_1}\to\cdots\to M_{l'_{c-1}}\to M_{l'_{c}}=M_{l},
\]
in the directed graph~$G(M,\alpha+\delta)$. Therefore~$l \in JM_{l_1}$ and 
\[
JN_{l_1} \subseteq JM_{l_1}.
\]
Therefore, by Theorem~\ref{theorem_1} and~\eqref{eq:4} we obtain
\[\pi_{l_1}(\operatorname{Pru}_{\alpha+\delta}(M)) = \bigcap_{l \in JM_{l_1}} M_l(-\delta)(-\alpha) \subseteq \bigcap_{l \in JN_{l_1}} N_l(-\delta)=\pi_{l_1}(\operatorname{Pru}_\delta(N)).\]
Since~$l_1$ and~$\delta$ were arbitrary,
\[
\forall \delta \geq 0, \forall l \in \{1,\dots,r\}, \quad \pi_l(\operatorname{Pru}_{\alpha+\delta}(M)) \subseteq \pi_l(\operatorname{Pru}_\delta(N)).
\]
Thus~$\forall \delta\geq 0, \Pru_{\alpha+\delta}(M) \text{ is an } \infty\text{-refinement of } \Pru_\delta(N)$, and we conclude that \[d_P(M,N) \leq d_B(M,N).\]
\end{proof}

\clearpage

\section*{Acknowledgments}
The author would like to thank his supervisors, Grégory Ginot and Vadim Lebovici, for their guidance, support, and many helpful discussions throughout the preparation of this work, as well as for their careful reading of earlier drafts and valuable suggestions that significantly improved the exposition.

\bibliography{references}        

\end{document}